\documentclass {elsarticle}

\journal{Systems and Control Letters}

\usepackage{graphicx}
\usepackage[all]{xy}
\usepackage{amsmath} 
\usepackage{amssymb}  
\usepackage{fourier}
\usepackage{color}
\newtheorem{thm}{Theorem}[section]
\newtheorem{prop}[thm]{Proposition}

\newtheorem{defn}[thm]{Definition}
\newtheorem{cor}[thm]{Corollary}
\def\qed{\hbox{\vrule width1.0ex height1.ex}\vspace{2mm}}
\newenvironment{proof}{\noindent{\bf Proof.}}{\qed \vskip 0.0ex}
\def\eqref#1{(\ref#1)}
\newcommand{\Rset}{\mathbb R}
\newcommand{\Cset}{\mathbb C}

\newcommand{\dd}{\mathrm d}

\newcommand{\e}{\mathrm e}
\newcommand{\ee}{\mathrm e}

\newcommand{\A}{\mathcal A}
\newcommand{\B}{\mathcal B}
\newcommand{\C}{\mathcal C}

\newcommand{\K}{\mathcal K}
\newcommand{\R}{\mathcal R}
\newcommand{\V}{\mathcal V}

\begin{document}
\begin{frontmatter}

\title{Exact observability and controllability for linear neutral type systems \tnoteref{mytitlenote}}
\tnotetext[mytitlenote]{Extended version of a paper presented at MTNS 14, Groningen, 2014.
This work was supported in part by PROMEP (Mexico) via "Poyecto de Redes" and by the  Polish Nat. Sci. Center,
grant N~514~238~438.
}
\author[authorlabel1]{R. Rabah\corref{mycorrespondingauthor}}
\ead{Rabah.Rabah@mines-nantes.fr}
\author[authorlabel2]{G. M. Sklyar}
\ead{sklar@univ.szczecin.pl}
\address[authorlabel1]{IRCCyN/\'Ecole des Mines, 4 rue Alfred Kastler, BP 20722,
44307 Nantes, France.}
\address[authorlabel2]{Inst. of Math., University of Szczecin,
Wielkopolska 15, 70451 Szczecin, Poland }
\cortext[mycorrespondingauthor]{Corresponding author}
\begin{abstract}
The problem of  exact observability  is analyzed for a wide class  of neutral type systems by
an infinite dimensional approach. The duality with the  exact controllability problem is the main tool. It is based
on an explicit expression of a neutral type system which corresponding to the abstract
adjoint system. A nontrivial  relation is obtained between the initial neutral system and the system obtained via
the adjoint abstract state operator. The characterization of the duality between  controllability and observability is deduced,
and then observability conditions are obtained.
\end{abstract}


\end{frontmatter}

\vskip -13cm \noindent \textcolor{blue}{http://dx.doi.org/10.1016/j.sysconle.2015.12.010} \vskip 13cm
\section{Introduction}
Approximate and spectral controllability  and the corresponding dual notions of observability  for delay systems of neutral type
were widely investigated at the end of the last century (see  books by \cite{Bensoussan_etal_1992}  and
 \cite{Salamon_Pitman_1984} and references therein). The duality between these notions for systems of neutral
type is not so trivial.  The main reason is that the dual or adjoint  system is not obtained directly by simple transposition.
It is necessary to consider the duality using some hereditary product proposed first for retarded systems and later
for neutral type systems
(see \cite{Delfour_Manitius_I_1980, Delfour_Manitius_II_1980, Salamon_IEEE1984} and  \cite{Salamon_Pitman_1984}
for example).
In this context, the important technique of the so called structural operator was used. It enables some
explicit formulations for duality between approximate controllability,  spectral controllability and the same notions for observability
and the characterizations of that concepts. We shall consider some of them in the context of our framework.

The infinite dimensional setting has been developed essentially for exact controllability and often for
neutral type systems without distributed delays. The exact observability problem has been less studied. 
In \cite{Metelskiy_Minyuk_2006, Khartovskii_2012} and \cite{Khartovskii_Pavlovskaya_2013}   an
approach is described based on the reconstruction of a part of the state for the case of a neutral type system with discrete delays.
A duality condition with null controllability is given. The time of  controllability (and  of possible reconstruction
of a part of the observed state) is estimated  sufficiently large, without more precision.

The present paper is concerned with exact observability  which is related to the notion of exact controllability
developed in the paper of the authors \cite{Rabah_Sklyar_2007} as an extension of other results obtained essentially
for neutral type systems with discrete delay \cite{OConnor_Tarn_1983b, Bartosiewicz_1984}. The semigroup approach used by the authors
in \cite{Rabah_Sklyar_2007} is based on the model introduced in \cite{Burns_Herdman_Stech_1983} in the product space $M_2$
(see the definition below, in this section).
 In the infinite dimensional  setting described in \cite{Rabah_Sklyar_2007}, exact  controllability means
 reachability of the operator domain because reachability of all the state space is
not possible by finite dimensional control. Hence, as it may be expected,  the dual notion of observability is
also adapted. Here, the approach using the structural operator is not used.
Considering the adjoint system in the operator form, that is in an  infinite dimensional
framework, we construct a transposed  neutral
type system corresponding to the adjoint system in Hilbert space.
This relation between the adjoint semigroup and the obtained neutral type system is different from that of the model
given in \cite{Burns_Herdman_Stech_1983}.

The notions of exact controllability and observability are important because they imply exponential stabilizability
or exponential convergences for possible estimators.

The results obtained in \cite{Rabah_Sklyar_2007} use the approach of moment problems and allows
the minimal time of exact controllability to be determinated. The  main contribution of our study is to specify how duality
may be used in a nontrivial context and to deduce  the characterization of exact  observability
and also the minimal time of observation.

\medskip

We consider the neutral type system given by the equation
\begin{equation}\label{eq:1}
\dot z(t) = A_{-1}\dot z(t-1) +
\int^0_{-1}\left [A_2(\theta)\dot z(t+\theta) + A_3(\theta) z(t+\theta)\right ]{\mathrm d}\theta
\end{equation}
where $A_{-1}$  is a constant $n\times n$-matrix,
and $A_2, A_3$ are $n\times n$-matrices whose
 elements belong to $L_2(-1,0).$

If we introduce the linear  operator $L : H^1([-1,0];\Rset^n) \longrightarrow \Rset^n$, defined by
\begin{equation}\label{op:L}
L f=\int_{-1}^0A_2(\theta)f'(\theta)+A_3(\theta)f(\theta)\dd \theta
\end{equation}
then the system may be written concisely as
$$
\dot z(t) = A_{-1}\dot z(t-1) +Lz_t, \qquad z_t(\theta)=z(t+\theta).
$$
This system may be represented, following the approach developed in \cite{Burns_Herdman_Stech_1983},
by an operator model in Hilbert space 
given by the equation
 \begin{equation}\label{eq:2}
   \dot x= {\mathcal A} x, \quad x(t)=\begin{pmatrix}
v(t)\\ z_t(\cdot)
\end{pmatrix},
\end{equation}
where ${\mathcal A}$ is the infinitesimal generator of a $C_0$-semigroup  $S(t)= \e^{\A t}$
given in the product space
$$
M_2=M_2(-1,0;\Rset^n)\stackrel{\mathrm{def}}{=} \Rset^n\times L_2(-1,0;\Rset^n)
$$
and defined by
\begin{equation}\label{2s}
{\mathcal A} x(t)={\mathcal A}
\begin{pmatrix}
v(t)\cr z_t(\cdot)
\end{pmatrix}  =
\begin{pmatrix}
Lz_t(\cdot)  \cr {\mathrm d}z_t(\theta)/ {\mathrm d}\theta
\end{pmatrix},
   \end{equation}
with the domain $D({\A})\subset M_2$ given by
\begin{equation}\label{2s_domain}
      \left \{
\begin{pmatrix}
 v \cr \varphi(\cdot)
\end{pmatrix}
: \varphi(\cdot) \in
   H^1, v=\varphi(0)-A_{-1}\varphi(-1)\right\},
\end{equation}
where $H^1=H^1([-1,0];\Rset^n)$.

 We consider the finite dimensional observation
\begin{equation}\label{output1}
 y(t)= \C x(t) 
\end{equation}
where $\C$ is a linear operator and $y(t) \in \Rset^p$ is a finite dimensional output. There are several
ways to design the output operator $\C$ \cite{Salamon_Springer_1983,Salamon_Pitman_1984,Metelskiy_Minyuk_2006}.
One of our goals in this paper is to investigate how to design
a minimal output operator like
\begin{equation}\label{output2}
 \C x(t)= Cz(t) \qquad \mbox{or} \qquad  \C x(t)= Cz(t-1),
\end{equation}
where $C$ is a $p\times n$ matrix. More general outputs, for example with several and/or distributed delays are not
considered in this paper. We want to use some results on exact controllability in order to analyze, by duality,
the exact observability property in the  infinite dimensional setting like, for example, in
\cite{Tucsnak_Weiss_2009}.

The operator $\C$ defined in (\ref{output2}) is linear but not bounded in $M_2$.
However, in both cases it is admissible in the following sense:
$$
\int_0^T\Vert\C S(t)x_0\Vert_{\Rset^n}^2\dd t \le \kappa^2 \|x_0\|_{M_2}^2, \qquad \forall
x_0 \in D({\mathcal A}),
$$
because it is bounded on $D(\A)$. We recall that if $x_0 \in D({\mathcal A})$ then
$S(t)x_0 \in D({\mathcal A}), \ t \ge 0$ (see for example \cite{Pazy_1983}).
 In fact, $\C$ is  admissible  in the
resolvent norm:
$$
\|x_0\|_{-1}=\left \Vert R(\lambda,\A)x_0\right\Vert=\left \Vert (\lambda I-\A)^{-1}x_0\right\Vert_{M_2},\quad
\lambda \in \rho(\A).
$$
This is a consequence of the fact that $\C$ is a closed operator and takes value
in a finite dimensional space (see \cite[Def.~4.3.1]{Tucsnak_Weiss_2009}
and comments on this Definition).
\begin{defn}\label{def1}
Let \ $\K$ be the output operator
$$
\K: M_2 \longrightarrow L_2(0,T;\Rset^p), \quad x_0 \longmapsto \K x_0=\C S(t)x_0.
$$
The system  (\ref{eq:1})  is said to be approximately observable (or  observable) if \ $\ker \K = \{0\}$ and
exactly observable (or continuously observable \cite{Salamon_Pitman_1984})  if
\begin{equation}\label{obsexact}
\Vert \K x_0 \Vert_{L_2}^2=\int_0^T\left\Vert \C S(t)x_0 \right \Vert_{\Rset^p}^2\dd t \ge \delta^2 \left\Vert x_0\right\Vert_{M_2}^2,
\end{equation}
for some constant $\delta$.
\end{defn}
\medskip

This is the classic definition.  In the case of a neutral type system with a finite dimensional
output (\ref{output2}) the exact observability in this sense  is not possible. It may be possible if we consider another
topology for the initial states $x_0$.\\
Unlike  approximate observability, which does not depend on the topology,
 exact observability depends essentially on the
topology in the space.
We can expect that, the given neutral type system is not exactly observable
if we consider $x_0 \in D(\A)$,  with the norm of the graph and no longer in the topology of $M_2$.
Taking in account the result on exact controllability, it seems that  (\ref{obsexact})
must be changed by taking a weaker norm for $x_0$, namely the resolvent norm
$\left \Vert (\lambda I- \A)^{-1}x_0\right\Vert$ and considering the extension of the operator $\K$
to the completion of the space with this norm.
In fact,  we obtain the observability in the initial norm but we need some delay in the observation in the general case.

Exact observability can be investigated directly, but another way is to use the
duality between exact observability and exact controllability. In \cite{Rabah_Sklyar_2007} the
conditions of exact controllability were given for the controlled system 
$$
\dot z(t) = A_{-1}\dot z(t-1) +Lz_t + Bu(t).
$$
In order to use the duality between observability and controllability, we need to compute the adjoint operator $\K^*$
in the duality  with respect to the pivot space $M_2$ in the embedding
\begin{equation}\label{embed}
X_1 \subset X=M_2 \subset X_{-1},
\end{equation}
where $X_1= D(\A)$ with the graph norm noted $\|x\|_1$ and
$X_{-1}$ is the completion of the space $M_2$ with respect to the resolvent norm
$\|x\|_{-1}=\left \|(\lambda I- \A)^{-1}x\right\|_{M_2}$.
The duality relation is
\begin{equation}\label{dual-rel}
\left \langle \K x_0, u(\cdot)\right \rangle_{L_2(0,T;\Rset^p)}
= \left \langle  x_0,\K^* u(\cdot)\right \rangle_{X_{1},X_{-1}^\dd},
 \end{equation}
where $X_{-1}^\dd$ is constructed as $X_{-1}$ with $\A^*$ instead of $\A$
(see  \cite{Tucsnak_Weiss_2009} for example).
Our purpose is to compute the adjoint operator
$$
\K^* : L_2(0,T;\Rset^p) \rightarrow X_{-1}^\dd.
$$
The abstract formulation is well known. Exact controllability is dual with exact observability
in the corresponding spaces with
the corresponding topologies.
It is expected that the operator $\K^*$ corresponds to a control operator for some adjoint system.

We then need the expression of  the adjoint state operator $\A^*$ and the corresponding adjoint
system in the same class: the class of neutral type systems. As it will be shown,  the situation is not so simple.
This is the object of Section~\ref{sec:adjsys}.
In Section~\ref{sec:duality} we return to the duality relation with the explicit
expression of the adjoint system after formulations of exact controllability results.
As the adjoint neutral
type system has a slightly different structure, we give an explicit relation between the new neutral
type system and the original one. After that we can give the expression of duality between exact
controllability and exact observability. This enables to formulate the characterization of exact observability and to give
the minimal time of observability. Some illustrative examples are given.

 For the sake of completeness, we recall  some results on approximate controllability
(from \cite{Salamon_Pitman_1984} and
\cite{Bartosiewicz_1984}) and formulate the duality with the corresponding notion of observability in our framework.
\section{The adjoint system}\label{sec:adjsys}
In this section we give the expression of the adjoint system corresponding
to the adjoint operator $\A^*$ as the operator $\A$ corresponds to the system (\ref{eq:1}).
Let us recall first  the expression of  the adjoint operator  ${\mathcal A}^*$ and its spectrum $\sigma(\A^*)$.
  \begin{prop} (\cite{Rabah_Sklyar_Rezounenko_2008}) \label{prop1}
The adjoint operator  ${\mathcal A}^{*}$ is
given by
\begin{equation}\label{adj1}
{\mathcal A}^{*}
\begin{pmatrix}
w\cr \psi (\cdot)
\end{pmatrix}
= \begin{pmatrix}
( A_2^{*}(0)w + \psi (0)
\cr -\frac{\dd [ \psi(\theta)+A_2^{*}(\theta)w ]}
{\dd \theta}
 +A_3^{*}(\theta)w
\end{pmatrix},
\end{equation}
with the domain $D({\A}^{*})$:
\begin{equation}\label{adj2}
\left\{ \left (w , \psi(\cdot)\right ) : \psi(\theta)+A_2^{*}(\theta)w \in H^1,A_{-1}^{*}\left ( A_2^{*}(0)w + \psi(0)\right)
 =\psi(-1)+A_2^{*}(-1)w \right \}.
\end{equation}
 $\sigma(\A^*)$ consists of   eigenvalues, roots of the equation $\det\Delta^{*}(\lambda) =0,$
where
\begin{equation}\label{Delta}
\Delta^*(\lambda) =
   \lambda I- \lambda {\mathrm e}^{-\lambda }A_{-1}^{*}-
  \int^0_{-1} {\mathrm e}^{\lambda s} \left  [ A_3^{*}(s) +
 \lambda A_2^{*}(s)\right ]{\mathrm d}s.
\end{equation}
\end{prop}
The adjoint operator $\A^*$ in (\ref{adj1}) seems to be different from a state  operator generated by a
neutral type system. However, we can construct a  system of neutral type corresponding, in some sense,
to the given adjoint operator.
\begin{thm}
 Let $x$ be a solution of  the abstract equation
\begin{equation}\label{eq:2.0*}
 \dot x=\A^*x, \quad x(t)=
\begin{pmatrix}
w(t)
\cr \psi_t(\theta)
\end{pmatrix}.
\end{equation}
Then the function $w(t)$ is the solution of the neutral type equation
 \begin{equation}\label{eq:adj0}
 \dot w(t+1) = A_{-1}^*\dot w(t) + \int_{-1}^0 \left [A_2^*(\tau)\dot w(t+1+\tau) +
 A_3^*(\tau) w(t+1+\tau)\right ]\dd \tau.
\end{equation}
\end{thm}
\begin{proof}
Our purpose is to find the corresponding neutral type equation in $\Rset^n$.
Equation (\ref{eq:2.0*}) may be written as
$$
\frac{\partial}{\partial t}
\begin{pmatrix}
 w(t)\cr\psi_t(\theta) 
\end{pmatrix}=
\begin{pmatrix} A_2^{*}(0)w(t) + \psi_t (0) \cr
 -\frac{\partial[\psi_t(\theta)+A_2^{*}(\theta)w(t)]}{\partial\theta }
 +A_3^{*}(\theta)w(t)
\end{pmatrix}
$$
 Let us put $r(\theta)=A_2^*(\theta)w+\psi(\theta)$ and
\begin{equation}\label{aux1}
  r(t,\theta)=A_2^*(\theta)w(t)+\psi_t(\theta)=A_2^*(\theta)w(t)+\psi_t(\theta).
\end{equation}
Then the operator $\A^*$ may be rewritten as
\begin{equation}\label{eq:adj1}
{\mathcal A}^{*}
 \begin{pmatrix}
w \cr
r(\theta) -A_2^*(\theta)w
\end{pmatrix}
= \begin{pmatrix} r(0) \cr - \frac{\dd r(\theta)}{\dd \theta }
 +A_3^{*}(\theta)w
\end{pmatrix},
\end{equation}
and the differential equation $\dot x=\A^*x$ as  a system of two equations:
\begin{equation}\label{eq:adj2}
\frac{\partial}{\partial t}
\begin{pmatrix} w(t) \cr r(t,\theta) -A_2^*(\theta)w(t)
\end{pmatrix}
=\begin{pmatrix}
r(t,0) \cr -\frac{\partial r(t,\theta)}{\partial \theta} +A_3^{*}(\theta)w(t)
\end{pmatrix}.
\end{equation}
The second   equation of this system may be written as a partial differential equation:
\begin{equation}\label{eq:partial}
\frac{\partial}{\partial t}r(t,\theta)+\frac{\partial}{\partial \theta}r(t,\theta)=
A_2^*(\theta)\dot w(t)+A_3^{*}(\theta)w(t).
\end{equation}
The general solution of this    equation is
 \begin{equation}\label{sol:partial}
r(t,\theta)= f(t-\theta) +   \int_0^\theta \left [A_2^*(\tau)\dot w(t-\theta+\tau) + A_3^*(\tau) w(t-\theta+\tau)\right ]\dd \tau,
\end{equation}
where $f(t-\theta)$ is the solution of the homogeneous equation obtained from (\ref{eq:partial}):
$$
\frac{\partial}{\partial t}r(t,\theta)+\frac{\partial}{\partial \theta}r(t,\theta)=0.
$$
and the second term
is a particular solution of (\ref{eq:partial}).

The first  equation of the system (\ref{eq:adj2}) gives
\begin{equation}\label{eq:ode1}
\dot w(t)= r(t,0).
\end{equation}
From (\ref{sol:partial}) (obtained from the second equation), putting $\theta=0$, we get with (\ref{eq:ode1})
\begin{equation}\label{eq:ode2}
\dot w(t)= r(t,0)=f(t).
\end{equation}
Then (\ref{sol:partial}) and (\ref{eq:ode2}) allow    $r(t,\theta)$ to be written as follows:
\begin{equation}\label{sol:partial2}
 r(t,\theta)= \dot w(t-\theta) + \int_0^\theta \left [A_2^*(\tau)\dot w(t-\theta+\tau) + A_3^*(\tau) w(t-\theta+\tau)\right ]\dd \tau.
\end{equation}
From the definition of the domain $D(\A^*)$ we obtain  $A_{-1}^*r(0)=r(-1)$. 
For the function $r(t,\theta)$, this condition reads
\begin{equation}\label{sol:partial3}
r(t,-1)=A_{-1}^*r(t,0)=A_{-1}^*\dot w(t)
\end{equation}
and by (\ref{sol:partial2}) we have
 \begin{equation}\label{sol:partial4}
r(t,-1)= \dot w(t+1) -\int_{-1}^0 \left [A_2^*(\tau)\dot w(t+1+\tau) + A_3^*(\tau) w(t+1+\tau)\right ]\dd \tau.
\end{equation}
Finally, from (\ref{sol:partial3}) and (\ref{sol:partial4}), we obtain the dual equation
\begin{equation}\label{eq:adj}
\dot w(t+1) = A_{-1}^*\dot w(t) + \int_{-1}^0 \left [A_2^*(\tau)\dot w(t+1+\tau) +
 A_3^*(\tau) w(t+1+\tau)\right ]\dd \tau.
\end{equation}
On the other hand the solution of equation (\ref{eq:adj2}) is
\begin{equation}\label{sol:sg}
\ee^{\A^* t}x_0=\begin{pmatrix}
 w(t) \cr \psi_t(\theta)
\end{pmatrix} =
\begin{pmatrix}
w(t) \cr r(t,\theta) -A_2^*(\theta)w(t)
\end{pmatrix},
\end{equation}
where $w(t)$ is the solution of equation (\ref{eq:adj}). If $x_0 \in X$ then it is a mild solution.
\end{proof}
This result may also be  formulated, by simple duality (transposition),  in the following way.
\begin{thm}\label{tildeA}
Let $x$ be a solution of the abstract equation
$$
 \dot x=\tilde \A x, \quad x(t)=\begin{pmatrix}
w(t) \cr \psi_t(\theta)
\end{pmatrix},
$$
where the operator $ \tilde \A$ is defined by
$$
{\tilde \A}
\begin{pmatrix}
w \cr \psi(\cdot)
\end{pmatrix}
= \begin{pmatrix}
 A_2(0)w + \psi (0) \cr - \frac{\dd [ \psi(\theta)+A_2(\theta)w]}{\dd \theta }
  +A_3(\theta)w
\end{pmatrix},
$$
with the domain
$$
 D(\tilde {\A})=  \left\{ \left (w , \psi(\cdot)\right ) : \psi(\theta)+A_2(\theta)w \in H^1,
 \left( A_{-1} A_2(0)- A_2(-1) \right) w
 =\psi(-1)-A_{-1}\psi(0) \right\}.
$$
Then the function $w(t)$ is the solution of the neutral type equation
\begin{equation}\label{eq:1t}
\dot w(t+1) = A_{-1}\dot w(t) +
\int_{-1}^0 \left [A_2(\tau)\dot w(t+1+\tau) +
 A_3(\tau) w(t+1+\tau)\right ]\dd \tau.
\end{equation}
\end{thm}
\bigskip
Let us now specify  the relation between the solutions of neutral type equations
(\ref{eq:1t}) and (\ref{eq:1}).
Let us put
$$
\begin{pmatrix}
w(t) \cr \psi_t(\theta)
\end{pmatrix} =
 \e^{\tilde {\A} t}\tilde x_0=
\e^{\tilde {\A} t} \begin{pmatrix}
w(0) \cr
\psi_0(\theta)
\end{pmatrix},
$$
and
$$
\begin{pmatrix}
v(t) \cr z_t(\theta)
\end{pmatrix}
 =
\begin{pmatrix}
w(t+1)-A_{-1}w(t) \cr w(t+1+\theta)
\end{pmatrix}=
\e^{\tilde {\A} t}
\begin{pmatrix}
v(0) \cr z_0(\theta)\end{pmatrix}
=\e^{\tilde  {\A} t}\xi_0 ,
$$
where $z_0(\theta)= w(t+1)$ and $v(0)=z_0(0)-A_{-1}z_0(-1)$. Our purpose is to
give the explicit relation between the initial conditions $\tilde x_0$ and
$\xi_0$:
$$
\tilde x_0 = \begin{pmatrix}
w(0) \cr \psi_0(\theta)
\end{pmatrix}, \qquad \xi_0= \begin{pmatrix}
v(0) \cr z_0(\theta)
\end{pmatrix}.
$$
The formal  relation between these vectors is
$$
\tilde x_0 = \begin{pmatrix}
w(0) \cr \psi_0(\theta)
\end{pmatrix}= F\xi_0= F
\begin{pmatrix}
w(1)-A_{-1}w(0) \cr w(\theta+1)
\end{pmatrix}.
$$
\begin{thm}\label{initial}
The operator $F$ representing the relation between initial conditions $\tilde x_0$ and $\xi_0$
corresponding to the neutral type systems (\ref{eq:1}) and (\ref{eq:1t}) is linear bounded
and bounded invertible from $X_1$ to $M_2$.
\end{thm}
\begin{proof}
Let us calculate the explicit expression for the linear operator $F$. From
(\ref{sol:partial2}) and (\ref{aux1}),  taking in account that we consider here the operator $\tilde \A$ instead
of $\A^*$,  we obtain
\begin{equation}\label{eq:xi}
\begin{array}{rcll}
  r(0,\theta)&= &\psi_0(\theta)+A_2(\theta)w(0)\\[2pt]
    &= &\dot w(-\theta)+\int_{0}^\theta \left [A_2(\tau)\dot w( \tau-\theta) +
 A_3(\tau) w(\tau-\theta)\right ]\dd \tau,
 \end{array}
\end{equation}
which can be written as
$$
 r(0,\theta)= \dot w(-\theta)+\int_{0}^{\theta} \left [A_2(\theta -s)\dot w(-s) +
 A_3(\theta-s) w(-s)\right ]\dd s.
$$
Putting $w(-s)=\int_0^s \dot w(-\sigma)\dd \sigma + w(0)$, we get
$$
\hskip -22em
r(0,\theta)-\int_{0}^\theta A_3(\theta-s)\dd s \cdot w(0)
$$
\begin{equation}\label{eq:xi0}%
=\dot w(-\theta) +\int_{0}^\theta \left [A_2(\theta -s)\dot w(-s) +
 A_3(\theta-s) \int_0^s \dot w(-\sigma)\dd \sigma \right ]\dd s.
\end{equation}
This may be represented by the expression
\begin{equation}\label{eq:xi1}
r(0,\theta)-\int_{0}^{\theta}A_3(\theta-s)\dd s \cdot w(0)=(I+\V) \dot w (-s),
\end{equation}
where $\V$ is the Volterra operator defined by
$$
\V\mu(\cdot)=\int_{0}^\theta \left [A_2(\theta -s) \mu(s) +
A_3(\theta-s)\int_0^s \mu(\sigma)\dd \sigma\right ]\dd s.
$$
The operator $\V$ is a compact linear operator from
$L_2(-1,0;\Rset^n)$ to $L_2(-1,0;\Rset^n)$
with a spectrum $\sigma(\V)=\{0\}$. This implies that the operator
$I+\V$ is bounded invertible on $L_2(-1,0;\Rset^n)$.

Let us now represent the operator $F$ as a composition of operators
according to the following commutative diagram
\begin{displaymath}
    \xymatrix{
{\begin{pmatrix}
v(0) \cr z_0(\theta)\end{pmatrix}}  \ar[d]_F  & \ar[r]^P & & \quad
{\begin{pmatrix}
w(0) \cr \dot w(-\theta)
\end{pmatrix}} \ar[d]^{Q} \\
{\begin{pmatrix}
w(0) \cr \psi_0(\theta)
\end{pmatrix}}     &  &\ar[l]^{R} &
\quad {\begin{pmatrix}
w(0) \cr (I+\V)\dot w(-s)
\end{pmatrix}}
}
\end{displaymath}

\medskip

\noindent where, as explained  above,
$\left(\begin{array}{c} v(0) \cr z_0(\theta)\end{array}\right)
 = \left(
\begin{array}{c}
w(1)-A_{-1}w(0) \cr w(\theta+1)\end{array}\right)$
and
$$
(I+\V)\dot w(-s)=r(0,\theta)+\int_{0}^{-\theta}A_3(\theta+s)\dd s \cdot w(0),
$$
where $r(0,\theta)$ is given as in (\ref{eq:xi}).
The operators $P :X_1 \longrightarrow M_2 $ and $R: M_2 \longrightarrow M_2$ are bounded invertible.
Moreover, as $I+\V: L_2(-1,0;\Rset^n) \longrightarrow  L_2 (-1,0;\Rset^n)$ is bounded invertible,
then $Q:M_2 \longrightarrow  M_2$ is  also bounded invertible.
\end{proof}
We also need the following property of the bounded operator $F^{-1}$.
\begin{prop}\label{prop}
 For $\lambda \ne \sigma(\tilde \A)$, the operator
$$
F^{-a}=F^{-1}(\lambda I -\tilde \A)
$$
can be extended
to a bounded (and bounded invertible) operator from $M_2$ to $M_2$.
\end{prop}
\begin{proof}
 We need to prove that
\begin{equation}\label{F*}
\|F^{-1}(\lambda I -\tilde \A)\tilde x_0\| \le C \|\tilde x_0\|, \qquad \tilde x_0 \in D(\tilde \A), C >0,
 \end{equation}
where $\|\cdot\|$ is the initial norm in $M_2$. Let $L_0$ and $D_0$ be the subspaces
$$
L_0 = \{(0, \psi(\cdot)) : \psi(\cdot) \in L_2(-1,0;\Rset^n)\}, \qquad D_0 = L_0 \cap D(\tilde \A).
$$
It is clear that $D_0$ is of finite co-dimension $n$, and this implies that it is enough to prove
the relation (\ref{F*}) for $\tilde x_0 \in D_0$. Let $\tilde x_0 = (0, \psi_0(\cdot)) \in D_0$.
The action of the operator $F^{-a}=F^{-1}(\lambda I -\tilde \A)$ may be decomposed according to the
following diagram
\medskip
$$
\begin{matrix}
\begin{pmatrix}
 0 \cr \psi_0(\cdot)
\end{pmatrix}\quad
\stackrel{(\lambda I -\tilde \A)}{\longrightarrow}
\quad
\begin{pmatrix}
 -\psi_0(0) \cr
\lambda \psi_0(\cdot) +\dot \psi_0(\cdot)
\end{pmatrix}\ \
 \stackrel{R^{-1}}{\longrightarrow}\
\begin{pmatrix}
 -\psi_0(0) \cr
\boldsymbol{\Psi_0}(\cdot)
\end{pmatrix}
\cr \cr
\downarrow^{F^{-a}} \hskip 17em \downarrow ^{Q^{-1}} \cr%
%
\end{matrix}
$$
\begin{displaymath} 
    \xymatrix{
{\begin{pmatrix}
w(1)- A_{-1}w(0)\cr w(\theta +1)
\end{pmatrix} }& &
 \ar[l]^{P^{-1}}& & & \hskip - 7em 
{\begin{pmatrix}
\psi_0(0) \cr \dot w(-\theta)
\end{pmatrix}}
}
\end{displaymath}
where
 $$
\boldsymbol{\Psi_0}(\theta)= \lambda \psi_0(\cdot)+\dot\psi_0(\cdot)+ \int_0\limits^{-\theta} A_3(\theta +s) \dd s \cdot \psi_0(0),
$$
and  the function $w(\cdot)$ is determinated from the
equation (obtained from (\ref{eq:xi})):
\begin{eqnarray*}
{\lambda \psi_0(\theta) +\dot\psi_0(\theta) - A_2(\theta)\psi_0(0)=}
\dot w(-\theta) + \int_{0}^{\theta} \left [A_2(\tau)\dot w(\tau-\theta) +
 A_3(\tau) w(\tau-\theta)\right ]\dd \tau,
\end{eqnarray*}
with the initial condition $w(0)=-\psi_0(0)$.
Integrating the last equation from $0$ to $-1-\theta$ and taking in account the initial condition, we obtain
\begin{eqnarray*}
\lefteqn{\psi_0(-1-\theta)+\lambda \int_0^{-1-\theta}\psi_0(\tau)\dd \tau - \int_0^{-1-\theta}A_2(\tau)\dd \tau\cdot\psi_0(0)
=}\cr
&& -w(1+\theta) + \int_0^{-1-\theta} \left (\int_0^\tau A_2(s)\dot w(s-\tau) +
 A_3(s) w(s-\tau)\dd s\right)\dd \tau.
\end{eqnarray*}
Using a transformation in the double integration and the initial condition, we get
\begin{eqnarray*}
\lefteqn{
\psi_0(-1-\theta)+\lambda \int_0^{-1-\theta}\psi_0(\tau)\dd \tau = }\\&&
 -w(1+\theta) + \int_0^{-1-\theta} \left( A_2(\tau) w(1+\theta+\tau)+\int_0^\tau A_3(s)w(s-\tau)\dd s \right )\dd \tau,
\end{eqnarray*}
and this can be written as
$
(I+\V_1) \psi_0(-1-\theta)=(-I+\V_2) w(1+\theta), \quad \theta \in [-1,0],
$
where $\V_1$ and $\V_2$ are Volterra operators from $L_2(-1,0;\Rset^n)$ to $L_2(-1,0;\Rset^n)$.
Both operators have spectra concentrated at  $\{0\}$. Then
$$
w(1+\theta) =(-I+\V_2)^{-1}(I+\V_1)\psi_0(-1-\theta)=\mathcal{W} \psi_0(-1-\theta),
$$
where $\mathcal{W}$ is a bounded invertible operator on $L_2(-1,0;\Rset^n)$.
This enables the final
expression for the operator $F^{-a}=F^{-1}(\lambda I -\tilde \A)$ on the set $D_0$ to be obtained:
 $$
F^{-a}
\begin{pmatrix}
 0 \cr \psi_0(\theta)
\end{pmatrix}
=
\begin{pmatrix}
\left.\mathcal{W} \psi_0(-1-\theta)\right |_{\theta=0}-
\left.A_{-1}\mathcal{W} \psi_0(-1-\theta)\right |_{\theta=-1} \cr   \mathcal{W} \psi_0(-1-\theta)
\end{pmatrix}
=
\begin{pmatrix}
w(1)-A_{-1}w(0) \cr   w(1+\theta)
\end{pmatrix}.
$$
Taking in account
$A_{-1}w(0)=-A_{-1} \psi_0(0)=- \psi_0(-1),$
we can rewrite
$$
\left. \mathcal{W} \psi_0(-1-\theta)\right|_{\theta=0}- \left. A_{-1}\mathcal{W} \psi_0(-1-\theta)\right|_{\theta=-1}
=\left.\mathcal{W}\psi_0(-1-\theta)\right |_{\theta=0} + \psi_0(-1).
$$
As    a Volterra operator is quasinilpotent, we can write:  $\left(-I+\V_2 \right)^{-1}= -\sum_{k=0}^\infty \V_2^k$ and
$\mathcal{W}=  -\sum_{k=0}^\infty \V_2^k \left(I+\V_1 \right)$. This gives
\begin{eqnarray*}
\mathcal{W}\psi_0(-1-\theta)|_{\theta=0} + 
\psi_0(-1)&=&\left.-\left(\sum_{k=1}^\infty \V_2^k(I+\V_1) + \V_1 \right)\psi_0(-1-\theta)\right|_{\theta=0} \cr
&=& G \psi_0(-1-\theta),
\end{eqnarray*}
and $G$ is a linear operator from a dense set of $L_2(-1,0;\mathbb{R}^n)$ to $\mathbb{R}^n$. 
Since equalities
$$
\begin{array}{l}
\left. \V_1 \varphi (-1-\theta)\right |_{\theta=0} = -\int\limits_{-1}^0 \varphi(\tau)\: d{\rm \tau},\cr
\left. \V_2 \varphi (-1-\theta)\right |_{\theta=0} = -\int\limits_{-1}^0
\left(A_2(\tau) \varphi(-1-\tau) \int\limits_{0}^\tau A_2(s) \varphi(\tau-s)\: d{\rm s} \right)\: d{\rm \tau}
\end{array}
$$
define bounded operators from $L_2(-1,0;\mathbb{R}^n)$ to $\mathbb{R}^n$, we conclude that the
operator $G$ can be  extended by continuity to a   linear bounded operator from $L_2(-1,0;\mathbb{R}^n)$ to $\mathbb{R}^n$.
As a consequence, the  operator $F^{-a}$ is extended to a bounded (in the norm of $M_2$) operator, defined on $L_0$ by the formula
\begin{equation}\label{eq:Fa01}
F^{-a}
\begin{pmatrix}
 0 \cr \psi(\theta)
\end{pmatrix}
=
\begin{pmatrix}
G \psi(-1-\theta) \cr   \mathcal{W} \psi(-1-\theta)
\end{pmatrix}.
\end{equation}
Let us observe that the subspace $L_0$ as well as its image $F^{-a} L_0$
have codimension  $n$ in the space $M_2$
 (a complement subspace for them is $\Rset^n\times \{0\} \subset M_2$).
Moreover, the mapping $L_0\mathop{\longrightarrow}\limits^{F^{-a}} F^{-a}L_0$ is bijective.
On the other hand,  the subspace  $D_0$ has codimension equals $n$ in $D(\tilde \A)$,
$F^{-a}$ is defined on $D(\tilde \A)$ by the formula
\begin{equation}\label{eq:Fa02}
F^{-a}=F^{-1}(\lambda I - \tilde \A),
\end{equation}
and the mapping $D(\tilde \A)\mathop{\longrightarrow}\limits^{F^{-a}} D(\tilde \A)$ is also bijective.
Besides, 
$
D(\tilde \A)+L_0=D(\tilde \A)+F^{-a}L_0=M_2.
$
Comparing all these facts, we conclude that the operator $F^{-a}$, given by formulas~(\ref{eq:Fa01}), (\ref{eq:Fa02}),
can be extended to a  bounded bijective operator on $M_2$.
\end{proof}
A direct consequence of this proposition is the following corollary.
\begin{cor}\label{cor1}
 For all $x \in D( \A)$, for $\lambda \ne \sigma(\tilde{\A})$,
we have
$$
c\|x\| \le \left\|(\lambda I-\tilde \A)^{-1}F x\right \| \le C \| x \|,
$$
where $\Vert \cdot \Vert$ is the norm of the space $M_2$.
\end{cor}
\section{The control system and duality}\label{sec:duality}%
Consider the controlled neutral type system
\begin{equation}\label{eq:1c}
\dot z(t) = A_{-1}\dot z(t-1) +Lz_t + Bu(t),
\end{equation}
where $u(t) \in L_2(0,T;\Rset^m)$ is a $m$-dimensional control vector-function.
This system may be represented  by an operator model in Hilbert space
given by the equation
 \begin{equation}\label{eq:2c}
   \dot x= {\mathcal A} x + \B u(t), \quad x(t)=
\begin{pmatrix}
v(t)\cr z_t(\cdot)
\end{pmatrix},
\end{equation}
where $\B u=(Bu,0)$ is linear and bounded from $\Rset^{m}$ to $M_2$. We can note that $\B$ is not
bounded from $M_2$ to $X_1$   because $\B u=(Bu,0) \notin D(\A)$ if
$Bu \neq 0$.
\subsection{Exact controllability}\label{sec:excontr}
Let us denote by $R_T \subset M_2$ the reachable subspace of the system (\ref{eq:2c}):
$$
R_T= \left \{ \R_Tu(\cdot)=\int_0^T \e^{\A t}\B u(t) \dd t:  u(t)  \in L_2(0,T;\Rset^m)\right \},
$$
where $\R_T: L_2 \longrightarrow   M_2$ is a linear bounded operator.
As  was pointed out in \cite{Ito_Tarn_1985} and \cite{Rabah_Sklyar_2007}, $R_T \subset D(\A)$ for all $T>0$.
This implies that exact controllability may be defined as follows.
\begin{defn}
 The system (\ref{eq:2c}) is exactly controllable if $R_T=D(\A)$.
\end{defn}
The abstract condition of exact controllability is (see \cite{Rudin_1975} for example)
\begin{equation}\label{crit-c}
\int_0^T\Vert \B^*\e^{\A^* t}x  \Vert^2_{\Rset^m} \dd t  \ge \delta^2 \Vert x \Vert^2_{X_{-1}^\dd}, \qquad \forall x_0 \in D(\A^*),
\end{equation}
which means that the operator $\R_T: L_2 \longrightarrow X_1$ is onto. Here the space $X_{-1}^\dd$
is the completion of the space $X=M_2$ with respect to the norm
$$
\|x\|_{X_{-1}^\dd}=\left \Vert(\lambda I-\A^*)^{-1}x \right\Vert_{M_2}, \quad \lambda \notin \sigma(\A^*).
$$
For the system (\ref{eq:1c}) the condition of exact controllability is given by the following theorem
(see \cite{Rabah_Sklyar_2007}).
\begin{thm}\label{th:contr}
 The system (\ref{eq:1c}) is exactly controllable at time $T$ if and only if, for all $\lambda \in \Cset$, the following
two conditions are verified
\begin{enumerate}
 \item[i)]

$\mathrm{rank}\left (
\lambda I-\lambda\ee^{-\lambda} A_{-1}-\int_{-1}^0\ee^{\lambda s}\left[\lambda A_2(s)+A_3(s) \right]\dd s
\quad  B\right )= n$,
\item[ii)]
$\mathrm{rank}\left (\lambda I-A_{-1}\quad  B\right )= n$.
\end{enumerate}
The time of controllability is $T>n_1(A_{-1}, B)$.
\end{thm}
\medskip
The integer $n_1(A_{-1}, B)$ is the controllability index of
the pair $(A_{-1}, B)$ (see \cite{wonham_1985}). If the delay is $h$, then the critical time is $n_1h$.

Let us now consider the dual notion of observability for the adjoint system.
The condition (\ref{crit-c}) is equivalent to the exact observability of the observed system
\begin{equation}\label{eq:2o}
 \left \{
\begin{array}{rcl}
 \dot x & =&\A^* x, \cr
y & = & \B^* x
\end{array}
\right.
\end{equation}
and the corresponding neutral type system is the system (\ref{eq:adj0}). Then the conditions (i)--(ii)
of Theorem~\ref{th:contr} are necessary and sufficient for the exact controllability of the adjoint system
(\ref{eq:2o}). But what is the corresponding property for the associate neutral type system (\ref{eq:adj})?
This question will be investigated  in the following paragraph.
\subsection{Duality}
Consider the transposed controlled neutral type system
\begin{equation}\label{eq:1c*}
\dot z(t) = A_{-1}^*\dot z(t-1) +L^*z_t + C^*u(t),
\end{equation}
where $L^*f= \int_{-1}^0A_2^*(\theta)f'(\theta)+A_3^*(\theta)f(\theta)\dd \theta$. Let $\A^\dag$
be the generator of the semigroup $\ee^{\A^{\dag} t}$
generated by this equation (\ref{eq:1c*}). We cannot consider $\A^*$ for this system because this operator
does not correspond directly to this system as  infinitesimal generator of the semigroup of solutions.
The domain $D({\A}^{\dag})$ of the
operator $\A^{\dag}$ is given by
$$
\left \{
\begin{pmatrix}
 v \cr z(\cdot)
\end{pmatrix} : z\in
   H^1([-1,0];\Cset^n), v=z(0)-A_{-1}^*z(-1)\right\}.
$$
The spectrum of $\A^\dag$ is $\sigma(\A^\dag)= \{\lambda : \Delta^*(\lambda)=0\}= \sigma(\A^*)$.
Let $X_1^\dag$ be $D({\mathcal A}^{\dag})$ with the norm
$$
\|x\|_{X_1^\dag}= \left \Vert(\lambda I-\A^\dag)x \right\Vert_{M_2}, \quad \lambda \notin \sigma(\A^\dag),
$$
which is equivalent to the graph norm. Consider now the reachability operator for this system
$$
\R_T^\dag u(t)= \int_0^T \e^{\A^\dag t}
\begin{pmatrix}
 C^* \cr 0
\end{pmatrix}
u(t) \dd t.
$$
From the properties of the operator $\R_T$, we can deduce that $\R_T^\dag$ is linear,
bounded from $L_2(0,T;\Rset^)$ to $X_1^\dag$. The exact controllability
for the system (\ref{eq:1c*}) can be formulated as
$$
R_T^\dag=\mathrm{Im}\,\R_T^\dag= X_1^\dag.
$$
The conditions of exact controllability for this system (\ref{eq:1c*}) may be obtained
directly from Theorem~\ref{th:contr}.

Let us now consider the corresponding space $X_{-1}^\dag$ of linear functionals on $X_1^\dag$
as the completion of the space $X=M_2$ with respect to the norm
$$
\|x\|_{X_{-1}^\dag}= \left \Vert(\lambda I-\A^\dag)^{-1}x\right \Vert_{M_2}, \quad \lambda \notin \sigma(\A^\dag).
$$
We then have  the embedding
\begin{equation}\label{embed+}
X_1^\dag \subset X=M_2 \subset X_{-1}^\dag.
\end{equation}
Then, for $x \in X_1^\dag$ and $y \in X_{-1}^\dag$, the functional acts as
\begin{equation}\label{embed+*}
\left \langle x,y  \right \rangle_{X_1^\dag, X_{-1}^{\dag \dd }} =
\left \langle  (\lambda I-\A^\dag)^{-1}x, (\lambda I-\A^{\dag *})^{-1} y  \right \rangle_{X },
\end{equation}
where $\A^{\dag *}$ is the adjoint of the operator $\A^{\dag}$ in $M_2$, and
the space $X_{-1}^{\dag \dd}$ is constructed as $X_{-1}^{\dag \dd }$ with
$\A^{\dag *}$ instead of $\A^{\dag}$ (see  \cite{Tucsnak_Weiss_2009} for example).

Let us note that the operator $\A^{\dag *}$ is in fact the operator $\tilde \A$
defined in Section~\ref{sec:adjsys} (see Theorem~\ref{tildeA} and later). We shall use
the properties obtained for this operator.

Let us now consider the adjoint $\R_T^{\dag *}$ of $\R_T^{\dag}$ with respect to the duality
induced by $X_1^\dag$ and  $X_{-1}^{\dag \dd }$ with the pivot space $X=M_2$.
Let $x_0 \in X$, then
$$
\begin{array}{lcl}
\left \langle \R_T^\dag u(\cdot),x_0 \right \rangle_{X_{1},X_{-1}^\dd}
&=& \left \langle \R_T^\dag u(\cdot),x_0 \right \rangle_{X} \cr
&=& \left \langle \displaystyle \int_0^T \ee^{\A^\dag t}  \begin{pmatrix} C^* \cr 0\end{pmatrix} u(t) \dd t
, x_0\right \rangle_{X}\\
&=& \displaystyle \int_0^T \left \langle    \begin{pmatrix} C^* \cr 0\end{pmatrix} u(t)
,   \ee^{\A^{\dag *} t}x_0\right \rangle_{X} \dd t.
\end{array}
$$
Suppose now that $x_0 =
\begin{pmatrix}
    w(0) \cr \psi_0(\theta)
\end{pmatrix}
\in D(\A^{\dag *})$.

Then, as a consequence of the
results in Section~\ref{sec:adjsys}, namely
from (\ref{sol:sg}) but for the operator $\A^{\dag *}=\tilde \A$, we obtain
$$
\ee^{\A^{\dag *} t} x_0=
\begin{pmatrix}
w(t) \cr \psi_t(\theta)
\end{pmatrix}
=
\begin{pmatrix}
w(t) \cr r(t,\theta -A_2(\theta)w(t)
\end{pmatrix}
, \quad t \ge 0.
$$
Hence,
$$
\begin{array}{rcl}
\left \langle \R_T^{\dag} u(\cdot), x_0 \right \rangle_{X}
&= & \int_0^T \left \langle u(t), Cw(t)\right \rangle_{\Rset^p} \dd t\\
 &= & \left \langle u(\cdot),\R_T^{\dag *}x_0 \right \rangle_{L_2}.
\end{array}
$$
On the other hand, we can write $x_0=F \xi_0$ (cf. Proposition~\ref{prop}), where
$$
\xi_0= \begin{pmatrix}
v(0) \cr z_0(\theta)
\end{pmatrix}
= \begin{pmatrix}
z_0(0)-A_{-1}z_0(-1)\cr z_0(\theta)
\end{pmatrix}
=
\begin{pmatrix}
w(1)-A_{-1}w(0)\cr w(\theta+1)
\end{pmatrix},
$$
and then
$$
\e^{\A t}\xi_0=\begin{pmatrix}
w(t+1)-A_{-1}w(t)\cr w(t+1+\theta)
\end{pmatrix}.
$$
Let $\K$ be the output operator introduced in Definition~\ref{def1}. Then
\begin{eqnarray*}
\lefteqn{\left \langle u(\cdot), \K\xi_0 \right \rangle_{L_2} =
\left \langle u(\cdot), \C \e^{\A t} \xi_0 \right \rangle_{L_2}= }\\
&&
\left \{
\begin{array}{{lcl}}
\int_0^T \left \langle u(t), Cw(t)\right \rangle_{\Rset^p} \dd t & \mbox{if}  & \C x(t)= Cz(t-1)\cr
\int_0^T \left \langle u(t), Cw(t+1)\right \rangle_{\Rset^p} \dd t  & \mbox{if}  & \C x(t)= Cz(t)
\end{array},
\right.
\end{eqnarray*}
which implies for all $x_0 \in X$:
\begin{equation}\label{dua-K}
\K F^{-1} x_0=
\left \{
\begin{array}{lcl}
\R_T^{\dag *}x_0 & \ \mbox{if} & \C x(t)= Cz(t-1),\cr
\e^{\A^{\dag *}}\R_T^{\dag *}x_0 & \ \mbox{if} &\C x(t)= Cz(t).
\end{array}
\right.
 \end{equation}
 \addtolength{\textheight}{-20ex}   


We can now formulate our main result on duality between exact controllability and exact observability.
\begin{thm}\label{th:dual}
1. The system~(\ref{eq:1}) with the output
$$
y(t) = \C x(t)= Cz(t-1)
$$
is exactly observable in the interval $[0,T]$, i.e.
$$
\Vert\K x_0\Vert_{L_2}^2 = \int_0^T\left\|\C \e^{\A t}x_0\right\|_{\Rset^p}^2\dd t \ge \delta^2\|x_0\|_{M_2}^2
$$
if and only if the adjoint system~(\ref{eq:1c*}) is exactly controllable at time $T$, i.e.
$$
R_T^\dag = \R_T^\dag \left(L_2(0,T;\Rset^p)\right)= X_1^\dag = D(\A^\dag).
$$
2.  If  $\, \det A_{-1} \ne 0$,  Assertion 1 of the theorem is verified for the output
$$
y(t) = \C x(t)= Cz(t),
$$
and for the same time $T$.
\end{thm}
\begin{proof}
Let us recall that the exact controllability of the system~(\ref{eq:1c*}) may be formulated by the equality
$\mathrm{Im}\,\R_T^\dag= X_1^\dag$. Then, taking in account the embedding (\ref{embed+}) and the duality product
(\ref{embed+*}), we can write the condition of exact controllability as  (see  \cite{Rudin_1975} for example):
\begin{equation}\label{crit-c1}
\| \R_T^{\dag *}x_0\|_{L_2}    \ge \delta   \left \|(\lambda I-\A^{\dag *})^{-1} x_0 \right\|_{M_2}
, \qquad \forall x_0 \in X.
\end{equation}
Let $\xi_0=F^{-1}x_0 \in D(\A)$, where $F:D(\A) \rightarrow M_2$ is the bounded invertible operator
defined in Section~\ref{sec:adjsys}. Let us remember that we have, from  (\ref{dua-K}):
$$
\R_T^{\dag *}x_0 =\K F^{-1}x_0= \K\xi_0.
$$
Then the inequality (\ref{crit-c1}) is equivalent to
\begin{equation}\label{crit-c2}
\| \K \xi_0\|_{L_2}   \ge \delta   \left \|(\lambda I-\A^{\dag *})^{-1} F \xi_0\right\|_{M_2}
, \qquad  \xi_0\in D(\A).
\end{equation}
Suppose now that the relation (\ref{crit-c2}) is verified for all $\xi_0\in D(\A)$, then
from Corollary~\ref{cor1} we obtain
$$
\| \K \xi_0\|_{L_2}   \ge \delta   \left \|(\lambda I-\A^{\dag *})^{-1} F \xi_0\right\|_{M_2}
\ge \underbrace{\delta c}_{= \delta_1} \|\xi_0\|_{M_2}, \ \xi_0\in D(\A).
$$
This inequality can be extended by continuity to $\xi_0 \in M_2$:
$$
\| \K \xi_0\|_{L_2}
\ge   \delta_1 \|\xi_0\|_{M_2}, \quad \forall  \xi_0 \in M_2
$$
Conversely, suppose that the preceding relation is verified. For $\xi_0 \in D(\A)$,
and from Corollary~\ref{cor1}, we get
$$
\|\xi_0\| \ge \frac{1}{C} \left \| \left(\lambda I-\A^{\dag *}\right)^{-1} F \xi_0\right\|_{M_2} ,
$$
and then
$$
\| \K \xi_0\|_{L_2} \ge \frac{\delta_1}{C} \left \|\left(\lambda I-\A^{\dag *}\right)^{-1} F \xi_0\right\|_{M_2}.
$$
This is the relation (\ref{crit-c2}) with $\delta=\delta_1/C$. As the relations
(\ref{crit-c1}) and  (\ref{crit-c2}) are equivalent, the first assertion of the theorem is proved.\\
To prove item 2 of the theorem, it is sufficient to remark that the condition
$\det A_{-1} \ne 0$ is equivalent to the fact that the operator $\e^{\A}$ is bounded invertible
($\e^{\A t}$ is a group), and then the relations
(\ref{crit-c1}) and  (\ref{crit-c2}) are equivalent.
\end{proof}
From this result and from Theorem~\ref{th:contr}
we can formulate the condition of exact observability.
\begin{thm}\label{th:exobs}
 1. The system~(\ref{eq:1})  with the output $y=Cz(t-1)$ is exactly observable over $[0,T]$ if and only if
\begin{enumerate}
\item[(i)] For all $\lambda \in \Cset$, $\mathrm{rank}\begin{pmatrix} \Delta^*(\lambda) & C^* \end{pmatrix}=n$,
where $\Delta^*(\lambda)$
is defined
in (\ref{Delta}).
\item[(ii)] For all $\lambda \in \Cset$, $\mathrm{rank}\begin{pmatrix} \lambda I - A_{-1}^* & C^* \end{pmatrix}=n$,
\item[(iii)] $T > n_1(A_{-1}^*, C^*)$, where $n_1$ is the 
index of controllability for the pair $(A_{-1}^*, C^*)$.
\end{enumerate}
2. If $\det A_{-1} \ne 0$, then  Assertion 1 is verified for the output $y(t)=Cz(t)$.
\end{thm}
\section{Approximate controllability and observability}
 Let us now formulate the result on approximate controllability and observability.
For more general duality relations between approximate controllability and observability of neutral type, we refer to the book
\cite{Salamon_Pitman_1984}. We give here a precise formulation in the light of our results on adjoint systems.
Let us first recall the definition
of approximate controllability.
\begin{defn}
The system~(\ref{eq:2c}) is approximately controllable at time $T$ if $\mathrm{cl\,}R_T=M_2$, where   $\mathrm{cl\,}R_T$ is
the closure of the attainable set $R_T$ at time $T$.
\end{defn}
Sometimes approximate controllability is defined as
$$
\mathrm{cl\,}\bigcup_{T>0}R_T=M_2,
$$
 however for neutral type systems this notion of approximate controllability (as exact controllability) means that there is an universal time
of controllability $T_0>0$  (see \cite{Salamon_Pitman_1984}), i.e.
such that:
$$
\mathrm{cl\,}R_{T_0}=\mathrm{cl\,}\bigcup_{T>0}R_T.
$$
  According to the relation (\ref{dua-K}) and the definition of observability we obtain the following result on duality
between approximate controllability and observability.
\begin{thm}\label{thm:appr_obs}
1. The system~(\ref{eq:1}) with the output
$
y(t) = \C x(t)= Cz(t-1)
$
is approximately observable in the interval $[0,T]$, i.e. $\K =\{0\}$
 if and only if the adjoint system~(\ref{eq:1c*}) is approximately  controllable at time $T$, i.e.
$
\mathrm{cl\,}R_T^\dag = M_2.
$
\\
2.  If  $\, \det A_{-1} \ne 0$,  Assertion 1 of the theorem is verified for the output
$
y(t) = \C x(t)= Cz(t),
$
and for the same time $T$.
\end{thm}
\begin{proof}
 The proof is a direct  consequence of the definitions and  (\ref{dua-K}).
\end{proof}
The conditions of approximate observability may be obtained from the conditions of approximate controllability by duality.
For our system, such conditions were obtained in \cite{Bartosiewicz_1984} in the space $W_1^2([-1,0];\Rset^n)$. In our notations,
this means that the reachability set $R_{T_0}$ for the system~(\ref{eq:2c}) is dense in $D(\A)$ with the norm of the graph. It is
equivalent to the density of $R_{T_0}$ in the space $M_2$.
 In \cite{Salamon_Pitman_1984}, it is shown that approximate observability and approximate controllability for such neutral type
systems are dual and this does not depend on the state space, and then duality holds in
the space $M_2$ \cite[Corollary 4.2.10]{Salamon_Pitman_1984}.

From the necessary and sufficient conditions in \cite[Th. 2]{Bartosiewicz_1984} and Theorem~\ref{thm:appr_obs},
it is easy to see that approximate observability holds if the following two conditions are verified:
\begin{enumerate}
\item  $\forall \lambda \in \Cset$, $\mathrm{rank}\begin{pmatrix} \Delta^*(\lambda) & C^* \end{pmatrix}=n$,
\item  $\mathrm{rank}\begin{pmatrix} A_{-1}^* & C^* \end{pmatrix}=n$.
\end{enumerate}
Note that the second condition is not necessary. The two conditions are verified when exact observability holds
(the second condition is the condition (ii) of Theorem~\ref{th:exobs} for the particular case  $\lambda =0$).
This emphasizes the difference between the concepts of exact and approximate observability.

\section{Examples}
Let us give some simple examples to illustrate our results.

\noindent\textit{Example 1.}
 Consider the system
$$
\dot z(t)=\dot z(t-1),
$$
where $z(t) \in \Rset^n, n > 1$, with two possible outputs
$$
y_0(t)=\C_0 x(t)= z(t), \qquad y_1(t)=\C_1 x(t)= z(t-1).
$$
The conditions of observability are verified, and the system is exactly observable for the output $y_0$ or
$y_1$.

\noindent\textit{Example 2.}
Consider the system
$$
\left \{
\begin{array}{rcl}
\dot z_1(t)&=&0\\
\dot  z_2(t)&=&\dot z_2(t-1) + z_1(t-1),
\end{array}
\right.
$$
where $z(t)=(z_1(t),z_2(t)) \in \Rset^2$,
with two possible output
$$
  y_0(t)=\C_0 x(t)= z(t), \qquad y_0(t)=\C_1 x(t)= z(t-1).
$$
The system with the output $y_1$ is exactly observable for the time $T>1$ and not observable for $T=1$.
The system with the output $y_0$ is not observable for any time $T>0$.
\section{Conclusion}
 For a large class of linear neutral type systems which include  distributed delays
we give the duality relation between exact controllability and
exact observability.  The characterization of exact observability is deduced.
\section*{References}

\def\cprime{$'$} \def\cprime{$'$}

\end{document}